\documentclass[final,1p]{elsarticle}
\usepackage{amssymb,amsfonts,amsmath,amsthm}
\newcommand{\ol}{\overline}

\newcommand{\veps}{\varepsilon}
\newcommand{\odin}[1]{{\mbox {\bf I}}_{\{#1\}}}

\newcommand{\fracd}[2]{\frac {\displaystyle #1}{\displaystyle #2 }}
\newcommand{\nn}{{\mathbb N}}
\newcommand{\rr}{{\mathbb R}}
\newcommand{\zz}{{\mathbb Z}}
\newcommand{\cald}{{\mathcal D}}
\newcommand{\calf}{{\mathcal F}}
\newcommand{\ms}{{\mathfrak S}}
\newcommand{\calr}{{\mathcal R}}
\newcommand{\bo}{{\bf 0}}
\newcommand{\ba}{{\bf a}}

\newcommand{\bx}{{\bf x}}
\newcommand{\by}{{\bf y}}
\newcommand{\bQ}{{\bf Q}}
\newcommand{\bX}{{\bf X}}
\newcommand{\bY}{{\bf Y}}
\newcommand{\beq}{\begin{eqnarray*}}
\newcommand{\feq}{\end{eqnarray*}}
\newcommand{\beqn}{\begin{eqnarray}}
\newcommand{\feqn}{\end{eqnarray}}

\newtheorem{theorem}{Theorem}
\makeatletter \@addtoreset{theorem}{section}\makeatother

\newtheorem{definition}[theorem]{Definition}
\newtheorem{lemma}[theorem]{Lemma}
\newtheorem{assume}[theorem]{Assumption}
\newtheorem*{theorem*}{Theorem}

\begin{document}

\begin{frontmatter}
\title{Multivariate linear recursions with Markov-dependent coefficients\tnoteref{label1}}
\tnotetext[label1]{Submitted September 21, 2009; Revised April 11, 2010}
\author{Diana Hay\fnref{a1}}
\fntext[a1]{Department of Mathematics, Iowa State University, Ames, IA 50011, USA}
\author{Reza Rastegar\fnref{a1}}
\author{Alexander Roitershtein\corref{cor1}\fnref{a1}}
\cortext[cor1]{Corresponding author. E-mail: roiterst@iastate.edu}
\begin{abstract}
We study a linear recursion with random Markov-dependent coefficients.
In a ``regular variation in, regular variation out" setup we show that its stationary solution
has a multivariate regularly varying distribution. This extends results previously established for i.i.d. coefficients.
\end{abstract}
\end{frontmatter}
\section{Introduction and statement of results}
\label{intro}
Let $\bQ_n$ be random $d$-vectors, $M_n$ random $d\times d$ matrices, and consider the recursion
\beqn
\label{leqn}
\bX_n=\bQ_n+M_n \bX_{n-1},\qquad \bX_n\in\rr^d,\,n\in\zz.
\feqn
This equation has been used to model the progression of real-world systems in discrete time, for example, in queuing theory \cite{brandt1} and financial models \cite{engle, mikosch}. See for instance \cite{review, embre-goldie, samorachev, vervaat} and
references therein for more examples.
\par
Let $I$ denote the $d\times d$ identity matrix and let $\Pi_n=M_0M_{-1}\cdots M_{-n}$
for $n\geq 0.$ It is well known (see for instance \cite{brandt}) that if
the sequence $(\bQ_n,M_n)_{n\in\zz}$ is stationary and ergodic,
and the following Assumption~\ref{assume} is imposed, then
for any $\bX_0$ series $\bX_n$ converges in distribution,
as $n\to\infty,$ to the random equilibrium
\beq
\bX=\bQ_0+\sum_{k=1}^\infty \Pi_{-k+1}\bQ_{-k},
\feq
which is the unique initial value making $(\bX_n)_{n \geq 0}$ into a stationary sequence.
\par
For $\bQ\in\rr^d$ define $\|\bQ\|=\max_{1\leq i \leq d} |\bQ(i)|$  and
let $\|M\|=\sup_{\bQ\in\rr^d,\|\bQ\|=1}\|M\bQ\|$ denote the corresponding operator norm
for a $d\times d$ matrix $M.$ The following condition ensures the existence and the uniqueness of the
stationary solution to \eqref{leqn}. The condition is also known to be close
to necessity (see \cite{dversa}).
\begin{assume}
\label{assume}
$\mbox{}$
\begin{itemize}
\item [(A1)]
$E\bigl(\log^+ \|M_0\|\bigr) <+\infty$ and $E\bigl(\log^+ \|\bQ_0\|\bigr) <+\infty,$ where $x^+:=\max\{x,0\}$
for $x\in\rr.$
\item [(A2)]
The top Lyapunov exponent $\lambda=\lim_{n\to\infty} \frac{1}{n}\log \|M_1M_2\cdots M_n\|$ is strictly negative.
\end{itemize}
\end{assume}
The stationary solution $\bX$ of the stochastic difference equation \eqref{leqn} has been studied by many authors.
Assuming the existence of a certain ``critical exponent" for $M_n,$ the distribution tails
$P(\bX\cdot \by>t)$ and $P(\bX\cdot \by<-t)$ for a deterministic vector $\by\in\rr^d$
were shown to be regularly varied (in fact, power tailed) in \cite{kesten-randeq} (for $d=1$ an alternative proof is given in
\cite{goldie}). Under different assumptions and for $d=1$ only, similar results for the
tails of $\bX$ were obtained in \cite{grey,trakai75}. The multivariate recursion \eqref{leqn} and tails of
its stationary solution $\bX$ were studied in \cite{guivarch3,guivarch,klupp} under conditions
similar to those of \cite{kesten-randeq}, and in \cite{maver, multi1} extending the one-dimensional
setup of \cite{grey,trakai75}. In all the works mentioned above, it is assumed that
$(\bQ_n,M_n)_{n\in\zz}$ is an i.i.d. sequence, and \cite{maver, multi1} suppose in addition
that the sequences $(\bQ_n)_{n\in\zz}$ and $(M_n)_{n\in\zz}$ are mutually independent.
\par
The goal of this paper is to extend the results of \cite{grey,trakai75}
to the case where $(Q_n,M_n)_{n \in \zz}$ are induced by a Markov chain.
The extension is desirable in many, especially financial, applications,
see for instance \cite{infor,msee}. We remark that in dimension one the results of \cite{kesten-randeq,goldie} (where $M_n$ is
dominant in determining the tail behavior of $\bX$) and \cite{grey,trakai75} (where $\bQ_n$ is dominant)
were extended to a Markovian setup in \cite{saporta,omar} and \cite{mrec7}, respectively.
\par
Let ${\bf I}_A$ denote the indicator function of the set $A,$ that is ${\bf I}_A$ is
one or zero according to whether the event $A$ occurs or not.
\begin{definition}
\label{inde}
The coefficients $(\bQ_n,M_n)_{n\in\zz}$ are said to be induced by a sequence of random variables
$(Z_n)_{n \in \zz},$ each valued in a finite set $\cald,$ if there exists a sequence of
independent random pairs $(\bQ_{n,i}, M_{n,i})_{n\in\zz,i\in\cald}$ with $\bQ_{n,i}\in\rr^d$ and $M_{n,i}$ being $d\times d$ matrices,
such that for a fixed $i\in\cald,$ $(\bQ_{n,i},M_{n,i})_{n\in\zz}$ are i.i.d and
\beqn
\label{gigs}
\bQ_n=\sum_{j\in\cald}\bQ_{n,j}\odin{Z_n=j}=\bQ_{n,Z_n}
\qquad
\mbox{\em and}
\qquad
M_n=\sum_{j\in\cald}M_{n,j}\odin{Z_n=j}=M_{n,Z_n}.
\feqn
\end{definition}
Notice that the randomness of the coefficients $(\bQ_n)_{n\in\zz}$ induced by a sequence $(Z_n)_{n\in\zz}$ is due to two factors:
\begin{itemize}
\item [ 1)] to the randomness of the underlying auxiliary process $(Z_n)_{n\in\zz},$ which can be thought as representative
of the ``state of the external world,"
\item[] and, given the value of $Z_n,$
\item [ 2)] to the ``intrinsic" randomness of characteristics of the system which is captured by the random pairs $(\bQ_{n,Z_n},M_{n,Z_n}).$
\end{itemize}
The independence of $\bQ_{n,i}$ and $M_{n,i}$ is not supposed in the above definition.
Note that when $(Z_n)_{n\in\zz}$ is a finite Markov chain, \eqref{gigs} defines a {\em Hidden Markov Model} (HMM).
See for instance \cite{technion} for a survey of HMM and their applications in various areas.
\par
We will further assume that the vectors $\bQ_{n,i}$ are {\em multivariate regularly varying}.
Heavy tailed HMM have been considered for instance in \cite{hmht}, see also references therein.
Recall that, for $\alpha \in\rr,$ a function $f:\rr\to\rr$ is {\em regularly varying of index $\alpha$}
if $f(t)=t^\alpha L(t)$ for some $L(t):\rr\to\rr$ such that $L(\lambda t)\sim L(t)$
for all $\lambda >0$ (that is $L(t)$ is {\em slowly varying}).
Here and henceforth $f(t)\sim g(t)$ (we will omit ``$t\to\infty$" as a rule) means $\lim_{t\to\infty} f(t)/g(t)=1.$
\par
Let $S^{d-1}$ denote the unit sphere in $\rr^d$ with respect
to the norm $\|\cdot\|.$
\begin{definition}
\label{rvar}
A random vector $\bQ \in\rr^d$ is said to be regularly varying with index $\alpha> 0$
if there exist a function $\ba:\rr\to\rr$ regularly varying with index $1/\alpha$ and
a finite Borel measure $\ms_\bQ$ on $S^{d-1}$ such that for all $t >0,$
\beqn
\label{dprop}
n P\bigl(\|\bQ\|>ta_n;\,\bQ/\|\bQ\|\in \cdot\bigr) \overset{v}{\to}_{n\to\infty}
=t^{-\alpha}\ms_\bQ(\cdot), \qquad \mbox{\em as}~n\to\infty,
\feqn
where $ \overset{v}{\to}$ denotes the vague convergence on $S^{d-1}$ and $a_n:=\ba(n).$
\par
We denote by $\calr_{d,\alpha,\ba}$ the set of all $d$-vectors regularly varying with index $\alpha,$
associated with function $\ba$ by \eqref{dprop}.
\end{definition}
Let $E$ be a locally compact Hausdorff topological space.
The vague convergence of measures $\nu_n\overset{v}{\to}_{n\to\infty}\nu$ for finite measures $\nu_n,$ $n\geq 0,$ and $\nu$
on $E$ means (see for instance Proposition~3.12 in \cite{extreme}) that $\limsup_{n\to\infty} \nu_n(K) \leq \nu(K)$ for all compact
$K\subset E$ and $\liminf_{n\to\infty} \nu_n(G) \geq \nu(G)$ for all relatively compact open sets $G\subset E.$
In this paper we consider vague convergence on either $S^{d-1}$ or $\ol \rr_0^d:=[-\infty,\infty]^d\backslash\{\bo\},$
where $\bf 0$ stands for the zero vector in $\rr^d.$ In both spaces the topology is inherited from $\rr^d$
(in the case of $\ol \rr_0$ by adding neighborhoods of infinity and removing neighborhoods of zero, see for instance
\cite{foundations} for more details) and can be defined using an appropriate metric making both into
a locally compact Polish (complete separable metric) space. A set $K\subset \ol \rr_0^d$ is relatively compact
if its closure does not include $\bo,$ which makes the space $\ol \rr_0^d$ especially useful when convergence of
regularly varying distributions is considered.
\par
The definition \eqref{dprop} is norm-independent and turns out to be
equivalent to the following condition (see for instance \cite{foundations,pointp} or \cite{eth}):
\begin{itemize}
\item [] There is a Radon measure $\nu$ on $\ol \rr_0^d$ such that
$n P\bigl(a_n^{-1}\bQ \in \cdot\bigr) \overset{v}{\to}_{n\to\infty} \nu(\cdot).$
The measure $\nu$ is referred to as the {\em measure of regular variation} associated with $(\bQ,\ba).$
\end{itemize}
The regular variation of a random vector $\bQ \in \rr^d$ implies
that its one-dimensional projections have regularly varying tails of a similar
structure. More precisely, if $\bQ$ is regularly varying then for any $\bx \in \rr^d,$
\beqn
\label{kprop}
\lim_{t\to\infty}\fracd{P\bigl(\bQ \cdot \bx>t\bigr)}{t^{-\alpha}L(t)}=w(\bx)
\feqn
for a slowly varying function $L$ and some $w(\bx):\rr^d\to\rr$ which is not identically zero.
The property \eqref{kprop} was used as a definition of regular variation in \cite{kesten-randeq}, and it
turns out to be equivalent to \eqref{dprop} for all non-integer $\alpha$ as well as for odd integers
provided that $\bQ$ has non-negative components with a positive probability \cite{chara}. The question whether
\eqref{kprop} and \eqref{dprop} are equivalent for even integers $\alpha$ in higher dimensions remains open.
\par
In this paper we impose the following conditions on the coefficients $(\bQ_n,M_n)_{n\in\zz}.$
\begin{assume}
\label{agg}
Let $(Z_n)_{n\in\zz}$ be an irreducible Markov chain with transition matrix $H$ and stationary distribution $\pi$
defined on a finite state space $\cald.$ Suppose that the coefficients $(\bQ_n,M_n)_{n\in\zz}$ in \eqref{leqn} are
induced by the stationary sequence $(Z_n)_{n\in\zz},$ Assumption~\ref{assume} is satisfied, and, in addition, there exist
a constant $\alpha>0$ and a regularly varying function $\ba:\rr\to\rr$ such that
\item [(A3)] For each $i\in\cald,$ $\bQ_{0,i}\in \calr_{d,\alpha,\ba}$ with an associated measure of regular variation
$\mu_i.$
\item [(A4)] $\Lambda(\beta):=\limsup_{n\to\infty} \frac{1}{n}\log E\bigl(\|\Pi_{-n}\|^\beta\bigr)<0$ for some
$\beta>\alpha.$ In particular,
\beqn
\label{ninep}
\mbox{\em There exists $m>0$ such that}~ E\bigl(\|\Pi_{-m}\|^\alpha\bigr)<1~\mbox{\em and}~E\bigl(\|\Pi_{-m}\|^\beta\bigr)<1.
\feqn
\end{assume}
The following theorem extends results of \cite{grey,trakai75,maver,mrec7} to multivariate
recursions of the form \eqref{leqn} with Markov-dependent coefficients.
\begin{theorem}
\label{vart}
Let Assumptions~\ref{agg} hold. Then $\bX \in \calr_{d,\alpha,\ba}$ with measure of regular variation $\mu_\bX(\cdot)=\sum_{k=-\infty}^0 E\bigl(\mu_{_{Z_k}} \circ \Pi_{k+1}^{-1}(\cdot)\bigr),$ where $\mu \circ \Pi^{-1}(\cdot)$ stands for $\mu\bigl(\{\bx:\Pi\bx\in\cdot\}\bigr).$
\end{theorem}
The theorem is an instance of the phenomenon ``regular variation in, regular variation out"
for the model \eqref{leqn}. We remark that the mechanisms leading
to regularly varying tails of $\bX$ are quite different in \cite{trakai75,grey} versus \cite{kesten-randeq,goldie}.
In the former case, Kesten's ``critical exponent" is not available, and therefore more explicit assumptions
about distribution of $\bQ_n$ are made. Then $\bQ_n$ dominates and creates cumulative effects,
namely $\bX$ turns out to be regularly varying as a sum of regularly varying terms $\Pi_{n+1}\bQ_n.$
The setup of Assumption~\ref{agg} is particularly appealing because a similar ``cumulative effect" enables one to gain insight
into the structure and fine properties of the sequence $(\bX_n)_{n\in\nn}$, in particular into the asymptotic behavior of both the partial sums as well
as multivariate extremes of $(\bX_n)_{n\in\nn},$ see for instance \cite{samorachev,borovs,ppps,maxima,ldp-rec,langevin4}.
\par
The proof of Theorem~\ref{vart} is included in Section~\ref{proofvart}, with the exception of the main technical lemma (Lemma~\ref{glemma} below)
whose proof is deferred to the Appendix. The proof combines ideas developed in \cite{grey}, \cite{maver},
and \cite{mrec7}. We notice that Grey conjectured in \cite{grey} that using his method it may be possible
to extend the results of \cite{maver} and rid of the assumption that $(\bQ_n)_{n\in\zz}$ and $(M_n)_{n\in\zz}$ are
independent. We accomplish here the program suggested by Grey, and in fact extend it further to
coefficients induced by a finite-state irreducible Markov chains.
\section{Proof of Theorem~\ref{vart}}
\label{proofvart}
The following result extends Lemma~2 in \cite{grey}
and the relation (2.4) in \cite{maver}. Notice, that in contrast to \cite{maver}
we do not assume that $\bQ$ and $M$ are independent.
\begin{lemma}
\label{glemma}
Let $\bY,\bQ$ be random $d$-vectors and $\Pi$ be a random $d\times d$ matrix such that
\begin{itemize}
\item[(i)] $\bQ$ is independent of the pair $(\bY,\Pi)$
\item[(ii)] For some constant $\alpha>0$ and regularly varying $\ba:\rr\to\rr,$
$\bY$ and $\bQ$ belong to $\calr_{d,\alpha,\ba}$ with associated measures of
regular variation measures $\nu$ and $\mu,$ respectively.
\item[(iii)] $E\bigl(\|\Pi\|^\beta\bigr)<\infty$ for some $\beta>\alpha.$
\end{itemize}
Then, $\bY +\Pi\bQ\in \calr_{d,\alpha,\ba}$ with associated measure of regular variation
$\nu (\cdot)  + E\bigl(\mu \circ \Pi^{-1}( \cdot \,)\bigr).$
\end{lemma}
The proof of Lemma~\ref{glemma} is deferred to the Appendix. The next lemma, which generalizes
Proposition~2.1 of \cite{mrec7}, is the key element of our proof of Theorem~\ref{vart}.
\begin{lemma}
\label{klemma}
Let Assumption~\ref{agg} hold. Fix an integer $k\leq -1$ and let $\bY_{k+1}\in \rr^d$ be a random vector such that
$\bY_{k+1} \in \sigma(Z_n,\bQ_n,M_n: n \geq k+1).$
Let $\bY_k=\bY_{k+1} +\Pi_{k+1}\bQ_k$ and write $\bY_k=\sum_{i\in\cald} \bY_{k,i}\odin{Z_k=i}.$
\par
Then, each vector $\bY_{k,i}$ belongs to $\calr_{d,\alpha,\ba}$ with associated measure of regular variation
$\nu_{k,i}:=E\bigl(\nu_{k+1,Z_{k+1}}(\cdot)\odin{Z_k=i}\bigr)+E\bigl(\mu_i\circ \Pi_{k+1}^{-1}(\cdot)\odin{Z_k=i}\bigr),$
and hence $\bY_k \in \calr_{d,\alpha,\ba}$ with associated measure of regular variation
$E\bigl(\nu_{k+1,Z_{k+1}}(\cdot)\bigr)+E\bigl(\mu_{_{Z_k}}\circ \Pi_{k+1}^{-1}(\cdot)\bigr).$
\end{lemma}
\begin{proof}
Since $P\bigl((\bQ_k,\Pi_{k+1},\bY_{k+1})\in \cdot|Z_{k+1}=i,Z_k=j\bigr)=
P\bigl((\bQ_{1,j},\Pi_{k+1,i},\bY_{k+1,i})\in\cdot),$
using Lemma~\ref{glemma} we obtain for Borel subsets $A\subset \ol \rr_0^d,$
\beq
&&
P\bigl(a_n^{-1}\bY_{k,i}\in A\bigr)=\sum_{j \in \cald} P\bigl(\bY_{k+1}+\Pi_{k+1}\bQ_k\in a_n A\bigl| Z_{k+1}=j,Z_k=i\bigr)\pi_iH(i,j)
\\
&&
\quad
=
\sum_{j \in \cald} P\bigl(\bY_{k+1,j}+\Pi_{k+1,j}\bQ_{k,i}\in a_n A\bigr)\pi_iH(i,j).
\\
&&
\qquad
=
\sum_{j \in \cald} \bigl[\nu_{k+1,j}(A)+E\bigl(\mu_i\circ \Pi_{k+1,j}^{-1}(A)\bigr)\bigr]\pi_iH(i,j)
\\
&&
\qquad
=
E\bigl(\nu_{k+1,Z_{k+1}}(A)\odin{Z_k=i}\bigr)+E\bigl(\mu_i\circ \Pi_{k+1}^{-1}(A)\odin{Z_k=i}\bigr).
\feq
The proof of the lemma is completed.
\end{proof}
We are now in position to complete the proof of Theorem~\ref{vart}. First we
introduce some notations. Throughout the rest of the paper:
\\
$\mbox{}$
\\
For a constant $\delta>0$ and a set $K$ (either in $S^{d-1}$ or $\ol \rr_0^d$), let $K^\delta$ denote
the closed $\delta$-neighborhood of $K,$ that is $K^\delta=\{\bx: \exists~\by\in K~\mbox{s.t.}~\|\bx-\by\|\leq \delta\}.$
For $\bx\in \rr^d\slash\{0\},$ let $\ol \bx$ denote its direction $\bx/\|\bx\|.$
For a set $G,$ let $\ol G$ denote its closure $\bigcap_{\delta>0} G^\delta.$
\\
$\mbox{}$
\\
The final step in the proof is similar to the corresponding
argument in \cite{maver}, and is reproduced here for the sake of completeness.
It follows from Lemma~\ref{klemma} that, for any $L\in\nn$ and Borel $A\subset \ol \rr_0^d,$
\beqn
\label{finitel}
\lim_{n\to\infty} nP\Bigl(\sum_{k=-L}^0 \Pi_{k+1}\bQ_k\in a_n A\Bigr)=\sum_{k=-L}^0
E\bigl(\mu_{_{Z_k}}\circ \Pi_{k+1}^{-1}(A)\bigr),
\feqn
while \cite[Theorem~1.4]{mrec7} yields with the help of \eqref{ninep} that for any constant $\delta>0,$
\beqn
\label{oned}
\lim_{L\to\infty} \limsup_{n\to\infty}\, n P\Bigl( \sum_{k=-\infty}^{-L-1}  \|\Pi_{k+1}\| \cdot \|\bQ_k \| > \delta a_n \Bigr)
=0.
\feqn
For a compact set $K\subset \ol \rr_0^d,$ we have\\
$
P\Bigl( \sum\limits_{k=-\infty}^0 \Pi_{k+1} \bQ_k  \in a_n K \Bigr)
\leq P\Bigl(\sum\limits_{k=-L}^0 \bQ_k\Pi_{k+1}  \in a_n K^\delta \Bigr)
+P\Bigl(\sum\limits_{k=-\infty}^{-L-1}  \|\Pi_{k+1}\bQ_k \|  > \delta a_n \Bigr).
$
Hence, $\limsup_{n\to \infty} nP\Bigl(\frac{1}{a_n}\sum_{k=-\infty}^0  \Pi_{k+1} \bQ_k  \in K \Bigr) \leq
\sum_{k=-L}^0 E \bigl(\mu_{_{Z_k}}\circ \Pi_{k+1}^{-1}(K^\delta) \bigr)$ in virtue of \eqref{finitel} and \eqref{oned}.
Letting then $\delta \to 0,$ we obtain
\beqn
\label{ap4}
\limsup_{n\to \infty} nP\Bigl( a_n^{-1}\sum_{k=-\infty}^0 \Pi_{k+1}\bQ_k \in K \Bigr)
\leq \sum_{k=-\infty}^0 E \bigl(\mu_{_{Z_k}}\circ \Pi_{k+1}^{-1}(K) \bigr).
\feqn
Let $G\subset \ol \rr_0^d$ be relatively compact and open.
Consider open relatively compact sets $G_k \subset\ol \rr_0^d,$ $k\in\nn,$
such that $G_k\subset \ol G_k \subset G_{k+1} \subset G.$
For any $m,L,$ there is $\veps >0$ such that
$
\Bigl\{\sum\limits_{k=-L}^0 \Pi_{k+1} \bQ_k  \in a_n G_m\Bigr\}\bigcup
\Bigl\{\Bigl\|\sum\limits_{k=-\infty}^{-L-1} \Pi_{k+1}\bQ_k \Bigl\|\leq \veps a_n \Bigr\}
\subset \Bigr\{\sum\limits_{k=-\infty}^0 \Pi_{k+1} \bQ_k  \in a_n G \Bigr\}.
$
\\
Therefore, with $\calf_0:=\sigma(M_n,Z_n:n\leq 0),$ we have for any $G_m,$
\beq
&&\liminf_{n\to \infty} n P\Bigl( a_n^{-1} \sum_{k=-\infty}^0 \Pi_{k+1} \bQ_k  \in G  \Bigr) =
\liminf_{n\to \infty} nE \Bigl[P\Bigl( a_n^{-1} \sum_{k=-\infty}^0 \Pi_{k+1} \bQ_k \in G | \calf_0 \Bigr)\Bigr]
\\
&&
\quad
= \liminf_{n \to \infty} E\Bigl[nP\Bigl(a_n^{-1}\sum_{k=-L}^0 \Pi_{k+1} \bQ_k \in G_m \Bigr|\calf_0 \Bigr)
P\Bigl(\Bigl\| \sum_{k=-\infty}^{-L-1} \Pi_{k+1} \bQ_k\Bigr\| \leq \veps a_n \Bigl| \calf_0 \Bigr)\Bigr] \\
&&
\quad
\geq
E \Bigl[\liminf_{n \to \infty}nP\Bigl(a_n^{-1} \sum_{k=-L}^0 \Pi_{k+1} \bQ_k \in G_m \Bigr|\calf_0\Bigr)
P\Bigl(\Bigl\|\sum_{k=-\infty}^{-L-1} \Pi_{k+1} \bQ_k \Bigr\|\leq a_n \veps \Bigl|\calf_0\Bigr)\Bigr],
\feq
where for the last inequality we used Fatou's lemma. Hence, \eqref{finitel} yields the lower bound
$\liminf\limits_{n\to \infty} nP\Bigl( a_n^{-1}\sum_{k=-\infty}^0 \Pi_{k+1}\bQ_k \in G \Bigr)
\geq \sum_{k=-L}^0 E \bigl(\mu_{_{Z_k}}\circ \Pi_{k+1}^{-1}(G_m) \bigr).$
Letting $m\to \infty$ and then $L\to \infty,$ $\liminf\limits_{n\to \infty} nP\Bigl( a_n^{-1}\sum_{k=-\infty}^0 \Pi_{k+1}\bQ_k \in G \Bigr)
\geq \sum_{k=-\infty}^0 E \bigl(\mu_{_{Z_k}}\circ \Pi_{k+1}^{-1}(G) \bigr).$
This bound along with \eqref{ap4} yield the claim of the theorem provided
that we have shown that $\mu_\bX(\cdot)=E \bigl(\mu_{_{Z_k}}\circ \Pi_{k+1}^{-1}(\cdot) \bigr)$ is a
Radon measure on $\ol \rr_0^d,$ that is (see for instance
Remark~3.3 in \cite{multi1}) $\mu_\bX(K)<\infty$ for any compact set $K\in \ol \rr_0^d.$
Toward this end notice that since $\epsilon_K:=\inf_{\bx\in K} \|\bx\|>0$ and
in virtue of {\em (A4)} of Assumption~\ref{agg},
\beq
&& \mu_\bX(K)\leq \sum_{k=-\infty}^0 E \Bigl[\sum_{i\in\cald} \mu_i\circ \Pi_{k+1}^{-1}(K)\Bigr]=
\sum_{k=-\infty}^0 E \Bigl[\sum_{i\in\cald} \mu_i\bigl(\{\bx:\Pi_{k+1}\bx\in K\}\bigr)\Bigr]
\\
&&
\quad
\leq
\sum_{k=-\infty}^0 E \Bigl[\sum_{i\in\cald} \mu_i\bigl(\{\bx: \|\bx\|\geq \epsilon_K \|\Pi_{k+1}\|^{-1}\}\bigr)\Bigr]=
\sum_{k=-\infty}^0 |\cald|\epsilon_K^{-\alpha}\,E\bigl(\|\Pi_{k+1}\|^\alpha\bigr)<\infty,
\feq
completing the proof of the theorem. \hfill \hfill \qed
\appendix
\section{Proof of Lemma~\ref{glemma}}
We need to show that for any compact set $K\subset S^{d-1},$
\beqn
\label{compact}
\limsup_{n\to\infty} nP\bigl(\|\bY +\Pi\bQ\|>ta_n, \ol{\bY +\Pi\bQ} \in K \bigr)\leq
t^{-\alpha}\bigl[\ms_\bY (K)  + E\bigl(\ms_\bQ \circ \Pi^{-1}( K \,)\bigr)\bigr]
\feqn
while for any open set $G\subset S^{d-1},$
\beqn
\label{open}
\liminf_{n\to\infty} nP\bigl(\|\bY +\Pi\bQ\|>ta_n, \ol{\bY +\Pi\bQ} \in G \bigr)\geq
t^{-\alpha}\bigl[\ms_\bY (G)  + E\bigl(\ms_\bQ \circ \Pi^{-1}( G \,)\bigr)\bigr]
\feqn
To this end, we will use a decomposition resembling the one exploited in \cite[Lemma~2]{grey} and \cite[Proposition~2.1]{mrec7}.
Namely, we fix $\veps>0$ and write for any Borel set $A \subset S^{d-1},$
$nP\bigl(\|\bY+\Pi \bQ\|>ta_n,\overline{\bY+\Pi\bQ}\in A\bigr)
=J_{t,A}^{(1)}(n)-J_{t,A}^{(2)}(n)+J_{t,A}^{(3)}(n)+J_{t,A}^{(4)}(n),$
where
\beq
J_{t,A}^{(1)}(n)&=& nP\bigl(\|\bY\|>t(1+\veps)a_n,\overline{\bY+\Pi\bQ}\in A\bigr),\\
J_{t,A}^{(2)}(n)&=& nP\bigl(\|\bY\|>(1+\veps)ta_n,\, \|\bY+\Pi \bQ\|\leq ta_n,\overline{\bY+\Pi\bQ}\in A\bigr)\\
J_{t,A}^{(3)}(n)&=& nP\bigl((1-\veps)ta_n<\|\bY\|\leq (1+\veps)ta_n, \|\bY+\Pi\bQ\|>ta_n,\overline{\bY+\Pi\bQ}\in A\bigr)\\
J_{t,A}^{(4)}(n)&=& nP\bigl(\|\bY\|\leq (1-\veps)ta_n,\|\bY+\Pi\bQ\|>ta_n,\overline{\bY+\Pi\bQ}\in A\bigr).
\feq
Fix a constant $\delta\in (0,1)$ and let $K\subset S^{d-1}$ be an arbitrary compact set. Then
$J_{t,K}^{(1)}(n)\leq nP\bigl(\|\bY\|>t(1+\veps)a_n,\ol \bY \in K^\delta\bigr)+
nP\bigl(\|\bY\|>t(1+\veps)a_n,\|\overline{\bY}-\overline{\bY+\Pi\bQ}\|>\delta \bigr).$
It is not hard to check that for any constant $\gamma>0$ and vectors $\bx,\by\in\ol \rr_0^d,$
\beqn
\label{normeq}
\|\overline{\by}-\overline{\bx+\by}\|>\gamma~\mbox{implies}~\|\bx\|>
\fracd{\gamma\|\by\|}{2+\gamma}.
\feqn
Thus $nP\bigl(\|\bY\|>t(1+\veps)a_n,\|\overline{\bY}-\overline{\bY+\Pi\bQ}\|>\delta \bigr) \leq
nP\Bigl(\|\bY\|>ta_n,\|\Pi\| \|\bQ\|>\fracd{\delta ta_n}{3}\Bigr)\leq
nP\bigl(\|\bY\|>ta_n\bigr)P\Bigl(\|\bQ\| \geq  \fracd{\delta ta_n^{\frac{\beta - \alpha}{2\beta}}}{3}\Bigr) +
nP\bigl(\|\Pi\| \geq a_n^{\frac{\alpha+\beta}{2\beta}}\bigr).$
Since $P\bigl(\|\Pi\| \geq a_n^{\frac{\alpha+\beta}{2\beta}}\bigr)\leq a_n^{-\frac{\alpha+\beta}{2}}\,E\bigl(\|\Pi\|^\beta\bigr),$
we have $\limsup\limits_{n\to\infty}\, nP\bigl(\|\bY\|>t(1+\veps)a_n,\|\ol \bY -\overline{\bY+\Pi\bQ}\|>\delta \bigr)=0.$
Thus
\beqn
\label{jeyf1}
\limsup\limits_{n\to\infty} J_{t,K}^{(1)}(n) \leq \lim\limits_{\delta\to 0}\limsup\limits_{n\to\infty}
\,nP\bigl(\|\bY\|>t(1+\veps)a_n,\ol \bY\in K^\delta\bigr)=t^{-\alpha} \ms_\bY(K).
\feqn
Since $J_{t,K}^{(2)}(n)\leq  nP\bigl(\|\Pi\|\geq ta_n^{\frac{\alpha + \beta}{2\beta}}\bigr)+
nP\bigl(\|\bY\|>(1+\veps)ta_n\bigr)P\bigl(\|\bQ\|\geq \veps ta_n^{\frac{\beta - \alpha}{2\beta}}\bigr),$ we have
\beqn
\label{jeyf2}
\limsup_{n\to\infty} J_{t,K}^{(2)}(n) =0.
\feqn
Next, $J_{t,K}^{(3)}(n)\leq  nP\bigl((1-\veps)t a_n <\|\bY\|\leq (1+\veps)ta_n\bigr)
\sim t^{-\alpha}\bigl[(1-\veps)^{-\alpha}-(1+\veps)^{-\alpha}\bigr].$
Hence
\beqn
\label{jeyf3}
\lim_{\veps\to 0} \limsup_{n\to\infty} J_{t,K}^{(3)}(n)=0.
\feqn
\par
Define $g_n(\bx,A)=nP\bigl(\|\bY+\Pi\bQ\|>ta_n,\overline{\bY+\Pi\bQ}\in K\bigr|\bY=\bx,\,\Pi=A\bigr).$
Fix constants $\rho>0$ and $\eta>0,$ and let $J_{t,K}^{(4)}(n)=J_{t,K}^{(4,1)}(n)+J_{t,K}^{(4,2)}(n)+J_{t,K}^{(4,3)}(n),$ where
\beq
J_{t,K}^{(4,1)} (n)&=&
E\bigl( g(\bY,\Pi){\bf I}_{\{\|\bY\|\leq(1-\veps)ta_n\}}\odin{\|\Pi\|>\rho} \bigr)\\
J_{t,K}^{(4,2)} (n)&=& E\bigl( g(\bY,\Pi)\odin{\|\bY\|\leq(1-\veps)ta_n}
{\bf I}_{\{\|\Pi\|\leq \rho\}}{\bf I}_{\{\|\bY\|>\eta\}} \bigr)\\
J_{t,K}^{(4,3)} (n)&=& E\bigl( g(\bY,\Pi)\odin{\|\bY\|\leq(1-\veps)ta_n}
{\bf I}_{\{\|\Pi\|\leq \rho\}} \odin{\|\bY\|\leq \eta} \bigr).
\feq
The first two terms tend to zero as $\eta$ and $\rho$ go to infinity. More precisely,
\beqn
\nonumber
\limsup_{n\to\infty} J_{t,K}^{(4,1)} (n)&\leq &
\limsup_{n\to\infty} E\Bigl( nP\bigl(\|\Pi\|\cdot \|\bQ\|>\veps ta_n\bigr|\Pi)\odin{\|\Pi\|>\rho} \bigr)
\\
\label{ineq1}
&=&
(\veps t)^{-\alpha}E\bigl(\|\Pi\|^\alpha \odin{\|\Pi\|>\rho}\bigr)\to_{\rho\to \infty} 0,\\
\nonumber
\limsup_{n\to\infty} J_{t,K}^{(4,2)} (n)&\leq &
\limsup_{n\to\infty} E\bigl( nP\bigl(\|\Pi\|\cdot \|\bQ\|>\veps ta_n\bigr|\Pi)\odin{\|\Pi\|\leq \rho}\odin{\|\bY\|>\eta} \bigr)
\\
\label{ineq4}
&\leq &
\rho^\alpha (\veps t)^{-\alpha} P\bigl(\|\bY\|>\eta\bigr)\to_{\eta\to \infty} 0.
\feqn
To show the asymptotic of $J_{t,K}^{(4,3)}(n)$ as $n$ goes to infinity write,
\beqn
\nonumber
J_{t,K}^{(4,3)}(n)&\leq&
nP\bigl(\eta+\|\Pi\bQ\|>ta_n,\overline{\Pi\bQ}\in K^\delta\bigr)
\\
\label{normeq1}
&+&
nP\bigl(\overline{\bY+\Pi\bQ}-\overline{\Pi\bQ}\|>\delta, \|\Pi\bQ\|\geq \veps ta_n,\|\bY\|\leq \eta\bigr).
\feqn
Applying the multivariate Breiman's lemma (see for instance \cite[Proposition~5.1]{lemma}) to the first term in the right-hand side of
the last inequality and \eqref{normeq} to the second, we obtain
$\limsup_{n\to\infty} J_{t,K}^{(4,3)}(n) \leq t^{-\alpha} E\bigl(\ms_\bQ \circ \Pi^{-1}(K)\bigr).$
Thus \eqref{compact} is implied by \eqref{jeyf1}-\eqref{normeq1}.
\par
It remains to show that \eqref{open} holds for any open set $G \subset S^{d-1}.$
According to \eqref{jeyf2}, $\limsup_{n\to\infty} J_{t,G}^{(2)}(n) \leq
\limsup_{n\to\infty} J_{t,\ol G}^{(2)}(n) =0.$
Let $G_k\subset S^{d-1},$ $k\in\nn$ be open sets such that $G_k\subset \ol G_k \subset G_{k+1} \subset G$ Let $\gamma_k=\frac{1}{2}\inf\{\|\bx-\by\|:\bx\in G_k,\by\in G^c\}.$ Then, $J_{t,G}^{(1)}(n)
\geq nP\bigl(\|\bY\|>t(1+\veps)a_n,\overline{\bY}\in G_k\bigr)-
nP\bigl(\|\bY\|>t(1+\veps)a_n,\|\overline{\bY}-\overline{\bY+\Pi\bQ}\|>\gamma_k \bigr).$
By \eqref{normeq}, $\liminf\limits_{n\to\infty} J_{t,G}^{(1)}(n) \geq \lim_{\veps\to 0}\liminf_{n\to\infty}
\,nP\bigl(\|\bY\|>t(1+\veps)a_n,\overline{\bY}\in G_k \bigr)=t^{-\alpha}\ms_\bY(G_k).$
Letting $k\to\infty$ we obtain $\liminf\limits_{n\to\infty} J_{t,G}^{(1)}(n) \geq t^{-\alpha} \ms_\bY(G).$
To conclude, observe that
\beq
J_{t,G}^{(4,3)}(n)
&\geq &nP\bigl(\|\Pi\bQ\|-\eta>ta_n,\overline{\Pi\bQ}\in G_k\bigr)-
nP\bigl(\rho\|\bQ\|-\eta>ta_n; \|\bQ\|\leq \eta \bigr)
\\
&-&
nP\bigl(\|\overline{\bY+\Pi\bQ}-\overline{\Pi\bQ}\|>\gamma_k, \|\Pi\bQ\|\geq \veps ta_n,\|\bY\|\leq \eta\bigr).
\feq
By \eqref{normeq},
$\liminf\limits_{n\to\infty} J_{t,G}^{(4,3)}(n) \geq t^{-\alpha} E\bigl(\ms_\bQ \circ \Pi^{-1}(G_k)\bigr).$
Letting $k\to\infty$ establishes \eqref{open}.
\section*{Acknowledgements}
We are very grateful to Krishna Athreya for the careful
reading of a preliminary draft of this paper and many helpful remarks and suggestions.
We would like to thank the anonymous Referee and the Associated Editor for helping us to significantly
improve the presentation of this paper.


\begin{thebibliography}{36}
\expandafter\ifx\csname natexlab\endcsname\relax\def\natexlab#1{#1}\fi
\providecommand{\bibinfo}[2]{#2}
\ifx\xfnm\relax \def\xfnm[#1]{\unskip,\space#1}\fi
\bibitem[{Brandt et~al.(1990)Brandt, Franken, and Lisek}]{brandt1}
\bibinfo{author}{A.~Brandt}, \bibinfo{author}{P.~Franken},
  \bibinfo{author}{B.~Lisek}, \bibinfo{title}{Stationary Stochastic Models},
  \bibinfo{publisher}{Wiley}, \bibinfo{address}{Chichester},
  \bibinfo{year}{1990}.
\bibitem[{Engle(1995)}]{engle}
\bibinfo{author}{R.~F. Engle}, \bibinfo{title}{ARCH. Selected Readings},
  \bibinfo{publisher}{Oxford Univ. Press}, \bibinfo{year}{1995}.
\bibitem[{Mikosch and Starica(2000)}]{mikosch}
\bibinfo{author}{T.~Mikosch}, \bibinfo{author}{C.~Starica},
\newblock \bibinfo{title}{Limit theory for the sample autocorrelations and
  extremes of a {GARCH}(1,1) process},
\newblock \bibinfo{journal}{Ann. Statist.} \bibinfo{volume}{28}
  (\bibinfo{year}{2000}) \bibinfo{pages}{1427--1451}. \bibinfo{note}{An
  extended version is available at www.math.ku.dk/slash/$\sim$mikosch}.
\bibitem[{Diaconis and Freedman(1999)}]{review}
\bibinfo{author}{P.~Diaconis}, \bibinfo{author}{D.~Freedman},
\newblock \bibinfo{title}{Iterated random functions},
\newblock \bibinfo{journal}{SIAM Rev.} \bibinfo{volume}{41}
  (\bibinfo{year}{1999}) \bibinfo{pages}{45--76}.
\bibitem[{Embrechts and Goldie(1994)}]{embre-goldie}
\bibinfo{author}{P.~Embrechts}, \bibinfo{author}{C.~M. Goldie},
\newblock \bibinfo{title}{Perpetuities and random equations},
\newblock in: \bibinfo{editor}{P.~Mandl}, \bibinfo{editor}{M.~Hu\u{s}kov\'{a}}
  (Eds.), \bibinfo{booktitle}{Asymptotic Statistics, 5th. Symp. (Prague, 1993),
  Contrib. Statist.}, \bibinfo{publisher}{Physica},
  \bibinfo{address}{Heidelberg}, \bibinfo{year}{1994}, pp.
  \bibinfo{pages}{75--86}.
\bibitem[{Rachev and Samorodnitsky(1995)}]{samorachev}
\bibinfo{author}{S.~T. Rachev}, \bibinfo{author}{G.~Samorodnitsky},
\newblock \bibinfo{title}{Limit laws for a stochastic process and random
  recursion arising in probabilistic modeling},
\newblock \bibinfo{journal}{Adv. in Appl. Probab.} \bibinfo{volume}{27}
  (\bibinfo{year}{1995}) \bibinfo{pages}{185--202}.
\bibitem[{Vervaat(1979)}]{vervaat}
\bibinfo{author}{W.~Vervaat},
\newblock \bibinfo{title}{On a stochastic difference equations and a
  representation of non-negative infinitely divisible random variables},
\newblock \bibinfo{journal}{Adv. in Appl. Probab.} \bibinfo{volume}{11}
  (\bibinfo{year}{1979}) \bibinfo{pages}{750--783}.
\bibitem[{Brandt(1986)}]{brandt}
\bibinfo{author}{A.~Brandt},
\newblock \bibinfo{title}{The stochastic equation ${Y}_{n+1}={A}_n{Y}_n+{B}_n$
  with stationary coefficients},
\newblock \bibinfo{journal}{Adv. in Appl. Probab.} \bibinfo{volume}{18}
  (\bibinfo{year}{1986}) \bibinfo{pages}{211--220}.
\bibitem[{Babillot et~al.(1997)Babillot, Bougerol, and Elie}]{dversa}
\bibinfo{author}{M.~Babillot}, \bibinfo{author}{P.~Bougerol},
  \bibinfo{author}{L.~Elie},
\newblock \bibinfo{title}{The random difference equation ${X}_n ={A}_n{X}_{n-1}
  + {B}_n$ in the critical case},
\newblock \bibinfo{journal}{Ann. Probab.} \bibinfo{volume}{25}
  (\bibinfo{year}{1997}) \bibinfo{pages}{478--493}.
\bibitem[{Kesten(1973)}]{kesten-randeq}
\bibinfo{author}{H.~Kesten},
\newblock \bibinfo{title}{Random difference equations and renewal theory for
  products of random matrices},
\newblock \bibinfo{journal}{Acta. Math.} \bibinfo{volume}{131}
  (\bibinfo{year}{1973}) \bibinfo{pages}{208--248}.
\bibitem[{Goldie(1991)}]{goldie}
\bibinfo{author}{C.~M. Goldie},
\newblock \bibinfo{title}{Implicit renewal theory and tails of solutions of
  random equations},
\newblock \bibinfo{journal}{Ann. Appl. Probab.} \bibinfo{volume}{1}
  (\bibinfo{year}{1991}) \bibinfo{pages}{126--166}.
\bibitem[{Grey(1994)}]{grey}
\bibinfo{author}{D.~R. Grey},
\newblock \bibinfo{title}{Regular variation in the tail of solutions of random
  difference equations},
\newblock \bibinfo{journal}{Ann. Appl. Probab.} \bibinfo{volume}{4}
  (\bibinfo{year}{1994}) \bibinfo{pages}{169--183}.
\bibitem[{Grincevi\v{c}ius(1975)}]{trakai75}
\bibinfo{author}{A.~K. Grincevi\v{c}ius},
\newblock \bibinfo{title}{One limit distribution for a random walk on the
  line},
\newblock \bibinfo{journal}{Lithuanian Math. J.} \bibinfo{volume}{15}
  (\bibinfo{year}{1975}) \bibinfo{pages}{580--589}.
\bibitem[{de~Saporta et~al.(2004)de~Saporta, Guivarc'h, and Page}]{guivarch3}
\bibinfo{author}{B.~de~Saporta}, \bibinfo{author}{Y.~Guivarc'h},
  \bibinfo{author}{E.~L. Page},
\newblock \bibinfo{title}{On the multidimensional stochastic equation
  ${Y}_{n+1}={A}_n{Y}_n+{B}_n$},
\newblock \bibinfo{journal}{C. R. Math. Acad. Sci. Paris} \bibinfo{volume}{339}
  (\bibinfo{year}{2004}) \bibinfo{pages}{499--502}.
\bibitem[{Guivarc'h(2006)}]{guivarch}
\bibinfo{author}{Y.~Guivarc'h},
\newblock \bibinfo{title}{Heavy tail properties of stationary solutions of
  multidimensional stochastic recursions},
\newblock in: \bibinfo{booktitle}{Dynamics \& Stochastics},
  volume~\bibinfo{volume}{48} of \textit{\bibinfo{series}{IMS Lecture Notes
  Monogr.}}, \bibinfo{publisher}{Inst. Math. Statist.},
  \bibinfo{address}{Beachwood, OH}, \bibinfo{year}{2006}, pp.
  \bibinfo{pages}{85--99}.
\bibitem[{Kl\"{u}ppelberg and Pergamenchtchikov(2004)}]{klupp}
\bibinfo{author}{C.~Kl\"{u}ppelberg}, \bibinfo{author}{S.~Pergamenchtchikov},
\newblock \bibinfo{title}{The tail of the stationary distribution of a random
  coefficient {AR}(q) model},
\newblock \bibinfo{journal}{Ann. Appl. Probab.} \bibinfo{volume}{14}
  (\bibinfo{year}{2004}) \bibinfo{pages}{971--1005}.
\bibitem[{Resnick and Willekens(1991)}]{maver}
\bibinfo{author}{S.~I. Resnick}, \bibinfo{author}{E.~Willekens},
\newblock \bibinfo{title}{Moving averages with random coefficients and random
  coefficient autoregressive models},
\newblock \bibinfo{journal}{Comm. Statist. Stochastic Models}
  \bibinfo{volume}{7} (\bibinfo{year}{1991}) \bibinfo{pages}{511--525}.
\bibitem[{Stelzer(2008)}]{multi1}
\bibinfo{author}{R.~Stelzer},
\newblock \bibinfo{title}{Multivariate {M}arkov-switching {ARMA} processes with
  regularly varying noise},
\newblock \bibinfo{journal}{J. Multivariate Anal.} \bibinfo{volume}{99}
  (\bibinfo{year}{2008}) \bibinfo{pages}{1177--1190}.
\bibitem[{Perrakis and Henin(1974)}]{infor}
\bibinfo{author}{S.~Perrakis}, \bibinfo{author}{C.~Henin},
\newblock \bibinfo{title}{The evaluation of risky investments with random
  timing of cash returns},
\newblock \bibinfo{journal}{Management Sci.} \bibinfo{volume}{21}
  (\bibinfo{year}{1974}) \bibinfo{pages}{79--86}.
\bibitem[{Collamore(2009)}]{msee}
\bibinfo{author}{J.~F. Collamore},
\newblock \bibinfo{title}{Random recurrence equations and ruin in a
  {M}arkov-dependent stochastic economic environment},
\newblock \bibinfo{journal}{Ann. Appl. Probab.} \bibinfo{volume}{19}
  (\bibinfo{year}{2009}) \bibinfo{pages}{1404--1458}.
\bibitem[{de~Saporta(2005)}]{saporta}
\bibinfo{author}{B.~de~Saporta},
\newblock \bibinfo{title}{Tails of the stationary solution of the stochastic
  equation ${Y}_{n+1}=a_n{Y}_n+b_n$ with {M}arkovian coefficients},
\newblock \bibinfo{journal}{Stochastic Process. Appl.} \bibinfo{volume}{115}
  (\bibinfo{year}{2005}) \bibinfo{pages}{1954--1978}.
\bibitem[{Roitershtein(2007)}]{omar}
\bibinfo{author}{A.~Roitershtein},
\newblock \bibinfo{title}{One-dimensional linear recursions with
  {M}arkov-dependent coefficients},
\newblock \bibinfo{journal}{Ann. Appl. Probab.} \bibinfo{volume}{17}
  (\bibinfo{year}{2007}) \bibinfo{pages}{572--608}.
\bibitem[{Ghosh et~al.(2010)Ghosh, Hay, Hirpara, Rastegar, Roitershtein,
  Schulteis, and Suh}]{mrec7}
\bibinfo{author}{A.~P. Ghosh}, \bibinfo{author}{D.~Hay},
  \bibinfo{author}{V.~Hirpara}, \bibinfo{author}{R.~Rastegar},
  \bibinfo{author}{A.~Roitershtein}, \bibinfo{author}{A.~Schulteis},
  \bibinfo{author}{J.~Suh}, \bibinfo{title}{Random linear recursions with
  dependent coefficients}, \bibinfo{year}{2010}. \bibinfo{note}{To appear in
  Statistics and Probability Letters}.
\bibitem[{Ephraim and Merhav(2002)}]{technion}
\bibinfo{author}{Y.~Ephraim}, \bibinfo{author}{N.~Merhav},
\newblock \bibinfo{title}{Hidden {M}arkov processes},
\newblock \bibinfo{journal}{IEEE Trans. Inform. Theory} \bibinfo{volume}{48}
  (\bibinfo{year}{2002}) \bibinfo{pages}{1518--1569}.
\bibitem[{Resnick and Subramanian(1998)}]{hmht}
\bibinfo{author}{S.~I. Resnick}, \bibinfo{author}{A.~Subramanian},
\newblock \bibinfo{title}{Heavy tailed hidden semi-{M}arkov models},
\newblock \bibinfo{journal}{Stoch. Models} \bibinfo{volume}{14}
  (\bibinfo{year}{1998}) \bibinfo{pages}{319--334}.
\bibitem[{Resnick(1987)}]{extreme}
\bibinfo{author}{S.~I. Resnick}, \bibinfo{title}{Extreme {V}alues, {R}egular
  {V}ariation and {P}oint {P}rocesses}, \bibinfo{publisher}{Springer},
  \bibinfo{address}{New York}, \bibinfo{year}{1987}.
\bibitem[{Resnick(2004)}]{foundations}
\bibinfo{author}{S.~I. Resnick},
\newblock \bibinfo{title}{On the foundations of multivariate heavy tail
  analysis},
\newblock \bibinfo{journal}{J. Appl. Probab.} \bibinfo{volume}{41}
  (\bibinfo{year}{2004}) \bibinfo{pages}{191--212}.
\bibitem[{Resnick(1986)}]{pointp}
\bibinfo{author}{S.~I. Resnick},
\newblock \bibinfo{title}{Point processes, regular variation and weak
  convergence},
\newblock \bibinfo{journal}{Adv. in Appl. Probab.} \bibinfo{volume}{18}
  (\bibinfo{year}{1986}) \bibinfo{pages}{66--138}.
\bibitem[{Lindskog(2004)}]{eth}
\bibinfo{author}{F.~Lindskog}, \bibinfo{title}{Multivariate extremes and
  regular variation for stochastic processes}, Ph.D. thesis,
  \bibinfo{address}{Z\"{u}rich, Switzerland}, \bibinfo{year}{2004}.
  \bibinfo{note}{Available from: www.e-collection.ethbib.ethz.ch/diss/}.
\bibitem[{Basrak et~al.(2000)Basrak, Davis, and Mikosch}]{chara}
\bibinfo{author}{B.~Basrak}, \bibinfo{author}{R.~A. Davis},
  \bibinfo{author}{T.~Mikosch},
\newblock \bibinfo{title}{A characterization of multivariate regular
  variation},
\newblock \bibinfo{journal}{Ann. Appl. Probab.} \bibinfo{volume}{12}
  (\bibinfo{year}{2000}) \bibinfo{pages}{908--920}.
\bibitem[{Borovkov and Borovkov(2008)}]{borovs}
\bibinfo{author}{A.~A. Borovkov}, \bibinfo{author}{K.~A. Borovkov},
  \bibinfo{title}{Asymptotic Analysis of Random Walks: Heavy-Tailed
  Distributions}, volume \bibinfo{volume}{118} of
  \textit{\bibinfo{series}{Encyclopedia of Mathematics and its Applications}},
  \bibinfo{publisher}{Cambridge University Press}, \bibinfo{address}{Cambridge
  UK}, \bibinfo{year}{2008}.
\bibitem[{Davis and Hsing(1995)}]{ppps}
\bibinfo{author}{R.~A. Davis}, \bibinfo{author}{T.~Hsing},
\newblock \bibinfo{title}{Point process and partial sum convergence for weakly
  dependent random variables with infinite variance},
\newblock \bibinfo{journal}{Ann. Probab.} \bibinfo{volume}{23}
  (\bibinfo{year}{1995}) \bibinfo{pages}{879--917}.
\bibitem[{de~Haan et~al.(1989)de~Haan, Resnick, Rootz\'{e}n, and
  de~Vries}]{maxima}
\bibinfo{author}{L.~de~Haan}, \bibinfo{author}{S.~I. Resnick},
  \bibinfo{author}{H.~Rootz\'{e}n}, \bibinfo{author}{C.~G. de~Vries},
\newblock \bibinfo{title}{Extremal behavior of solutions to a stochastic
  difference equation with applications to {ARCH} processes},
\newblock \bibinfo{journal}{Stochastic Process. Appl.} \bibinfo{volume}{32}
  (\bibinfo{year}{1989}) \bibinfo{pages}{213--224}.
\bibitem[{Konstantinides and Mikosch(2005)}]{ldp-rec}
\bibinfo{author}{D.~G. Konstantinides}, \bibinfo{author}{T.~Mikosch},
\newblock \bibinfo{title}{Large deviations for solutions to stochastic
  recurrence equations with heavy-tailed innovations},
\newblock \bibinfo{journal}{Ann. Probab.} \bibinfo{volume}{33}
  (\bibinfo{year}{2005}) \bibinfo{pages}{1992--2035}.
\bibitem[{Rastegar et~al.(2010)Rastegar, Roytershteyn, Roitershtein, and
  Suh}]{langevin4}
\bibinfo{author}{R.~Rastegar}, \bibinfo{author}{V.~Roytershteyn},
  \bibinfo{author}{A.~Roitershtein}, \bibinfo{author}{J.~Suh},
  \bibinfo{title}{Discrete-time {L}angevin motion of a particle in a {G}ibbsian
  random potential}, \bibinfo{year}{2010}. \bibinfo{note}{The preprint is
  available at
  hhtp:$\slash\slash$www.public.iastate.edu$\slash\sim$roiterst$\slash$papers$%
\slash$langevin4.pdf}.
\bibitem[{Basrak et~al.(2002)Basrak, Davis, and Mikosch}]{lemma}
\bibinfo{author}{B.~Basrak}, \bibinfo{author}{R.~A. Davis},
  \bibinfo{author}{T.~Mikosch},
\newblock \bibinfo{title}{Regular variation of {GARCH} processes},
\newblock \bibinfo{journal}{Stochastic Process. Appl.} \bibinfo{volume}{99}
  (\bibinfo{year}{2002}) \bibinfo{pages}{95--115}.

\end{thebibliography}
\end{document}